\documentclass{amsart}
\usepackage{verbatim}
\usepackage{rotating} 
\usepackage{amssymb}
\usepackage{mathrsfs,amsfonts,amsmath,amssymb,epsfig,amscd,xy,amsthm,pb-diagram,hyperref} 

\newtheorem{theorem}{Theorem}[section]
\newtheorem{proposition}[theorem]{Proposition}

\newtheorem{lemma}[theorem]{Lemma}

\newtheorem{corollary}[theorem]{Corollary}
\newtheorem{question}[theorem]{Question}
\newtheorem{definition}[theorem]{Definition}

\theoremstyle{plain}

\theoremstyle{remark}

\newtheorem{remark}[theorem]{Remark}

\newcommand{\C}{{\mathbb C}}

\newcommand{\Q}{{\mathbb Q}}

\newcommand{\Z}{{\mathbb Z}}
\newcommand{\N}{{\mathbb N}}

\newcommand{\Qbar}{\bar{\Q}}

\DeclareMathOperator{\den}{den}

\DeclareMathOperator{\lcm}{lcm}

\DeclareMathOperator{\Tr}{Tr}
\DeclareMathOperator{\Norm}{N}

\DeclareMathOperator{\ord}{ord}

\newcommand{\bP}{{\mathbb P}}

\newcommand{\bfu}{{\mathbf u}}

\newcommand{\cO}{\mathcal{O}}

\newcommand{\scrL}{\mathscr{L}}

\DeclareMathOperator{\characteristic}{char}

\DeclareMathOperator{\Trace}{Trace}

\author{Jason P.~Bell}
\address{
Jason P.~Bell\\
University of Waterloo\\
Department of Pure Mathematics\\
Waterloo, Ontario, Canada N2L 3G1}
\email{jpbell@uwaterloo.ca}

\thanks{Jason Bell was supported by NSERC grant RGPIN-2016-03632; Khoa Nguyen was supported by NSERC grant  RGPIN-2018-03770 and CRC tier-2 research stipend 950-231716. We are grateful to Andrew Granville and the anonymous referee for helpful comments that improve the paper.}
\author{Khoa D.~Nguyen}
\address{
Khoa D.~Nguyen \\
Department of Mathematics and Statistics\\
University of Calgary\\
AB T2N 1N4, Canada
}
\email{dangkhoa.nguyen@ucalgary.ca}

\author{Umberto Zannier}
\address{
Umberto Zannier\\
Scuola Normale Superiore, Classe di Scienze Matematiche e Naturali, Pisa, Italy
}
\email{umberto.zannier@sns.it}

\keywords{D-finite power series, heights, rational functions, }
\subjclass[2010]{Primary: 13F25, 12H05. Secondary: 11G50}

\begin{document}
	\title[D-finiteness, rationality, and height]{D-finiteness, rationality, and height II: lower bounds over a set of positive density}
	\date{November 2022}
	\begin{abstract}		 We consider D-finite power series $\displaystyle f(z)=\sum_{n\geq 0} a_n z^n$ with coefficients in 
	a number field $K$.  We show that there is a dichotomy governing the behaviour of $h(a_n)$ as a function of $n$, where $h$ is the absolute logarithmic Weil height. As an immediate consequence of our results, we 
have that	
	either $f(z)$ is rational or $h(a_n)>[K:\mathbb{Q}]^{-1}\cdot \log(n)+O(1)$ for $n$ in a set of positive upper density and this is best possible when $K=\mathbb{Q}$.  
			\end{abstract}
	
	\maketitle
	
	\section{Introduction}\label{sec:intro}
	Within number theory, combinatorics, and other disciplines, it is often of use to understand the asymptotics of a sequence of complex numbers; that is, to describe how the sequence is growing.  This approach is especially fruitful when trying to gain insight into the behaviour of integer-valued sequences satisfying recurrences.  When considering non-integer-valued sequences, however, knowing the actual size (or, more precisely, modulus) is typically of less importance and one often uses other measures that serve as proxies for how the values of a sequence $(a_n)$ increase in complexity with $n$.  For number fields, one generally uses the notion of \emph{heights} (see \S\ref{sec:height} for relevant definitions) as a measure of the complexity of elements of the field, and interprets numbers with smaller height as being simpler, or less complex, than numbers with larger height.  
	
	In the paper \cite{BNZ20_DF}, improving results of 
	van der Poorten-Shparlinski~\cite{vdPS96_OL} and Bell-Chen~\cite{BC17_PS}, we considered exactly these sorts of questions, looking at the growth of heights of coefficients of multivariate D-finite power series with coefficients in a number field $K$.  These are power series $F(z_1,\ldots ,z_d)\in K[[z_1,\ldots ,z_d]]$ in variables $z_1,\ldots ,z_d$ that satisfy non-trivial homogeneous linear differential equations with polynomial coefficients with respect to each differential operator $\partial/\partial z_i$.  In this case we found that if the heights of the coefficients of $F$ do not grow too quickly then $F$ is necessarily a rational function of a special form.  In simpler terms, if the complexity of the coefficients of $F$ grows sufficiently slowly then $F$ is well behaved.  More precisely, we showed that if 
	$$f(z_1,\ldots ,z_d) = \sum a_{n_1,\ldots ,n_d} z_1^{n_1}\cdots z_d^{n_d}\in K[[z_1,\ldots,z_n]]$$ and the logarithmic Weil height of $a_{n_1,\ldots ,n_d}$ is $o(\log\Vert{\bf n}\Vert)$, with $\Vert{\bf n}\Vert:=n_1+\cdots +n_d$, as $\Vert{\bf n}\Vert\to\infty$ then $f(z_1,\ldots ,z_d)$ is the power series expansion of a rational function and the coefficients $a_{n_1,\ldots,n_d}$ belong to a finite set.   
This result suggests, then, that there is a type of ``height boundary'' for coefficients of multivariate D-finite series, given by the function $\log\,||{\bf n}||$, where if the coefficients have heights that are small compared to this boundary function then the power series is rational 
and the coefficients satisfy $h(a_{n_1,\ldots,n_d})=O(1)$.

	This paper provides a more refined result toward a complete classification of the height growth of
	the coefficients of a D-finite univariate power series in $\Qbar[z]$. 
	These are exactly power series solutions  of linear homogeneous differential equations with coefficients in $\Qbar(z)$ and play an important role in many areas of mathematics. While the coefficients of the power series of a rational function in $\Qbar(z)$ satisfy a linear recurrence relation with constant coefficients and are rather well understood, the coefficients of a D-finite series 
satisfy a linear recurrence relation with polynomial coefficients and remain quite mysterious (see the questions in Section~\ref{sec:conclude}). 
	Although the authors obtained most of the results in this paper in the fall of 2019 shortly after \cite{BNZ20_DF}, its release has been delayed due to various reasons. In the meantime,  \cite{BNZ20_DF} has various natural motivations toward certain results in number theory and dynamics. In arithmetic dynamics, \cite{BNZ20_DF} motivates a height gap conjecture in work of Bell, Hu, Ghioca, and Satriano  \cite{BHS20_HG,BGS21_DU} concerning the quantity
	$\limsup_{n\to\infty} h(f(\Phi^n(x)))/\log n$,
	where $\Phi:\ X\dashrightarrow X$ and $f:\ X\dashrightarrow\bP^1$ are rational maps
	on quasi-projective varieties defined over $\Qbar$ and $x\in X(\Qbar)$ is such that
	the forward orbit of $x$ under $\Phi$ is well defined and avoids the indeterminacy locus of $f$. In the theory of Mahler functions, Adamczewski, Bell, and Smertnig \cite{ABS20_AH} provide a complete classification of the possible height growth of the coefficients of a Mahler function. The paper \cite{BNZ20_DF} together with earlier work of Perelli and Zannier \cite{PZ84_OR,Zan96_OP} motivate \cite{BN21_AA} in which an analogue of a conjecture of Ruzsa for polynomials over finite fields is established. Finally \cite{BNZ20_DF,BN21_AA} and the discussion related to results in this paper 
in part give motivation to Dimitrov for his spectacular solution \cite{Dim19_AP} of the Schinzel-Zassenhauss conjecture from the 1960s.

	For an algebraic number $\alpha$, we define its denominator $\den(\alpha)$ to be the smallest positive integer $d$ such that $d\alpha$ is an algebraic integer. We recall that a subset $S$ of $\N$ is said to have positive (upper) density if 
	$$\limsup_{n\to\infty}\frac{\vert S\cap[1,n]\vert}{n}>0;$$
	otherwise $S$ is said to have zero density. Finally, we use Vinogradov notation $\gg$ and $\ll$ for sequences of non-negative real numbers, where $a_n\gg b_n$ for $n$ in a set $T$ means that there is a positive constant $C$ such that $Ca_n\ge   b_n $ for $n\in T$, with $\ll$ defined analogously.   
	Our main result gives a coarse classification of the growth of $h(a_n)$, where $$f(z)=\displaystyle\sum_{n=0}^{\infty}a_nz^n$$ is a D-finite power series with algebraic coefficients, $n$ belongs to a set of positive density, and where $h:K \to [0,\infty)$ is the absolute logarithmic Weil height.  
	
	\begin{theorem}\label{thm:main thm}
	Let $K$ be a number field and let $f(z)=\sum_n  a_nz^n\in K[[z]]$ be a D-finite power series. Let $r$ be the radius of convergence of $f$. Then, we have:
	\begin{itemize}
		\item [(a)] If $r\in\{0,\infty\}$ and $f$ is not a polynomial then $h(a_n)=O(n\log n)$ as $n$ tends to infinity and $h(a_n)\gg n\log n$ on a set of positive upper density.
		\item [(b)] If $r\notin\{0,\infty\}$ then at least one of the following holds:
		\begin{itemize}		
		\item [(i)] $h(a_n)\gg n$ on a set of positive upper density;
		
		\item [(ii)] $\den(a_n)\gg n$, and hence
					$\displaystyle h(a_n)>\displaystyle\frac{1}{[K:\Q]}\log n+O(1)$, 
					on a set of positive upper density;
					
		\item [(iii)] $f(z)$ is a rational function whose denominator divides $(1-z^M)^N$ for some nonnegative integers $M,N$ (i.e. all of the finite poles of $f$ are roots of unity).
		\end{itemize}	
	\end{itemize}
	\end{theorem}
	
	\begin{remark}
	In fact, we prove a more precise version of Theorem~\ref{thm:main thm} (see Theorem~\ref{thm:slightly stronger} and Lemma~\ref{lem:kappa from beta}). When (iii) holds, we have $a_n=P(n)$ for all large $n$ in an appropriate arithmetic progression where $P$ is a polynomial. In this case, we still have $h(a_n)>\log n+O(1)$ on a set of positive density unless $P$ is constant, in which case the sequence $(a_n)_{n\geq 0}$ is eventually periodic.
	\end{remark}
	
	Theorem~\ref{thm:main thm} immediately yields the following result. 
	\begin{corollary}\label{cor:cor of main thm}
	Let $K$ and $f$ be as in Theorem~\ref{thm:main thm}. Suppose that $f$ is not a rational function, then $$\displaystyle h(a_n)> \frac{1}{[K:\Q]}\log n+O(1)$$ on a set of positive upper density.
	\end{corollary}

	\begin{remark}
	It follows from \cite{BNZ20_DF} that $\displaystyle\limsup h(a_n)/\log n>0$ or in other words there exist $C>0$ and an infinite $S\subseteq\N$ depending on $f$ such that $h(a_n)>C\log n$ for $n\in S$. Corollary~\ref{cor:cor of main thm} improves this in two aspects: one may require that $S$ has positive density and take the uniform constant $C=1/[K:\Q]$ up to the additional $O(1)$ term. When $K=\Q$, we have the lower bound $h(a_n)> \log n+O(1)$ on a set of positive density.  This is best possible thanks to examples such as
	$$\log(1+z^m)=z^m-\frac{z^{2m}}{2}+\frac{z^{3m}}{3}-\cdots,$$
	which show both that we cannot improve the constant and that the density of the set $S$ can be arbitrarily close to zero.	
	For a general number field $K$, it is plausible that the lower bound $\log n+O(1)$ remains valid (i.e., the factor $1/[K:\Q]$ in the statement of Corollary~\ref{cor:cor of main thm}~(ii) can be improved to $1$) but our method has not been able to yield this.  We leave this as a question for future work (see Question \ref{Q1}).  
	\end{remark}

	The outline of this paper is as follows. In \S\ref{prelim}, we give an overview of Weil heights and 
	the technical ingredients in the proof of Theorem~\ref{thm:main thm}. 
	This includes the use of Hankel determinants, P\'olya's inequality, and diophantine approximation results when approximating a D-finite series by a rational function. In \S\ref{sec:proof}, we first prove Theorem~\ref{thm:slightly stronger} which is a slightly more precise variant of Theorem~\ref{thm:main thm} then obtain the proof of the latter. Besides the technical results introduced in \S\ref{prelim}, our method involves the construction of an auxiliary polynomial $P(z)$ and work with the power series
	$\sum P(n)a_nz^n$ instead of the original series
	$ f(z)=\sum a_nz^n$. Right after the completion of this paper, a variant of
	our method has been extended to obtain a general criterion for the P\'olya-Carlson dichotomy and its application to the Artin-Mazur zeta function associated to endomorphisms on positive characteristic tori \cite{BGNS22_AG}. 
	In \S\ref{sec:conclude} we provide some comments to results in the paper. This includes a demonstration that one cannot, in a certain sense, strengthen Theorem~\ref{thm:main thm} in the style of Theorem~\ref{thm:slightly stronger} and several open problems for future work. These problems include a complete classification of the possible growth of $h(a_n)$ for a D-finite series $\sum a_nz^n\in\Qbar[[z]]$ as well as a similar long standing open problem in the theory of Siegel E-functions.

	\section{Preliminary results}\label{prelim}
	\subsection{Heights}\label{sec:height}
	We give a quick overview of heights along with the results we will need. For a number field $K$,
	we let $M_K^\infty$ denote the set of archimedean places (equivalence classes of archimedean absolute values) of $K$ and we let 
	$M_K^0$ denote the set of finite places.  
  	We write $M_K=M_K^\infty\cup M_K^0$.
	 For every place
  	$w\in M_K$, let $K_w$ denote the completion of 
  	$K$ with respect to $w$ and we let
  	$d(w)$ denote the quantity $[K_w:\Q_v]$ where $v$ is the restriction of
  	$w$ to $\Q$. 
	
	We now follow the treatment given in \cite[Chapter~1]{BG06_HI}:
  	for every $w\in M_K$ with restriction $v$ on $\Q$, we may always take $\vert\cdot \vert_v$ to be either the ordinary Euclidean absolute value or the $p$-adic absolute value for some prime $p$.
  	We can then normalize $\vert \cdot\vert_w$ by defining
  	$$\vert x\vert_w = \vert \Norm_{K_w/\Q_v}(x) \vert_v^{1/[K:\Q]}.$$

  	Let $m\in\N$. Given a vector $\bfu=(u_0,\ldots,u_m)\in K^{m+1}\setminus\{\mathbf 0\}$ and $w\in M_K$,
  	we can now define $$\vert \bfu\vert_w:=\displaystyle\max_{0\leq i\leq m} \vert u_i\vert_w.$$
  	In particular, we can now define heights for points in projective varieties.  
	For $P \in \bP^m(\Qbar)$, let $K$ be a number
  	field such that $P$ has a representative 
  	$\bfu\in K^{m+1}\setminus\{\mathbf 0\}$
  	and define:
  	$$H(P)=\prod_{w\in M_K} \vert \bfu\vert_w.$$
  	Define $h(P)=\log (H(P))$. Finally, for $\alpha\in \Qbar$, 
  	we write $H(\alpha)=H([\alpha:1])$
  	and $h(\alpha)=\log(H(\alpha))$. 
	
	Then the height function defined has the following properties (see \cite[Part~B]{HS00_DG} or \cite[Chapters~1--2]{BG06_HI} for proofs of these facts):
	  	\begin{enumerate}
  	\item [(i)] For $a\in \Qbar^*$ and $m\in\Z$, $h(a^m)=\vert m\vert h(a)$.
  	\item [(ii)] For $r\in \N$ and $a_1,\ldots,a_r\in\Qbar$, $h(a_1+\ldots+a_r)\leq h(a_1)+\ldots+h(a_r)+\log r$.
  	\item [(iii)] For every $a,b\in\Qbar$, $h(ab)\leq h(a)+h(b)+\log 2$. Hence if $b\neq 0$ then $h(ab)\geq h(a)-h(b)-\log 2$.
  	\end{enumerate}

	\subsection{Hankel determinants}\label{subsec:Hankel} 
	Throughout this subsection, let $k$ be a field, let $(a_n)_{n\in\N_0}$ be a $k$-valued sequence, and let $$f(z)=\displaystyle\sum_{n=0}^\infty a_nz^n\in k[[z]]$$ be the associated formal power series.  We assume tacitly that $f$ is nonzero.
We let $\ord$ denote the order function on $k[[z]]$: for $g\in k[[z]]$, $\ord(g)$ is the largest integer $m$ such that $z^m\mid g$ (with the convention that $\ord(0)=\infty$).  We use the notation $g=O(z^m)$ to mean that  $\ord g\ge m$.

For integers $\ell,m\ge 0$ we define the Hankel matrix  $H_{\ell,m}(f)$  and Hankel determinant $\Delta_{\ell,m}(f)$ by:
\begin{equation*}
H_{\ell,m}(f)=\begin{pmatrix} a_\ell & a_{\ell+1} & \ldots & a_{\ell+m} \\
 a_{\ell+1} & a_{\ell+2} &\ldots  & a_{\ell+m+1}  \\
 \ldots  \\
 a_{\ell+m} & a_{\ell+m+1} &\ldots  & a_{\ell+2m} \end{pmatrix},\qquad \Delta_{\ell,m}(f):=\det H_{\ell,m}(f).
\end{equation*}
 
Hankel determinants are of importance due to their connection with rationality of power series.
We begin with a lemma, which illustrates this connection in terms of approximation by rational power series.

\begin{lemma} \label{lem:H1} 
Let $\ell$ and $m$ be nonnegative integers. Then $\Delta_{\ell,m}(f)=0$ if and only if there exist polynomials $P(z),Q(z)\in k[z]$  with
$\deg P\le \ell+m-1$, $Q(z)\neq 0$, $\deg Q\le m$ and $P(z)-Q(z)f(z)=O(z^{\ell+2m+1})$.
\end{lemma}

\begin{proof} Suppose $\Delta_{\ell,m}=0$, so the columns of  $H_{\ell,m}$ are linearly dependent over $k$; i.e., there exist $q_0,\ldots ,q_m\in k$, not all $0$, and with $$q_0a_{\ell+r+m}+q_1a_{\ell+r+m-1}+\ldots +q_ma_{\ell+r}=0$$ for $r=0,1,\ldots ,m$.  

Setting $Q(z):=q_0+q_1z+\ldots +q_mz^m$, we have $Q(z)\neq 0$ and 
\begin{equation*}
Q(z)f(z)=\sum_{s=0}^\infty c_sz^s,\qquad  c_s=q_0a_s+q_{1}a_{s-1}+\cdots +q_ma_{s-m},
\end{equation*}
with the convention that  $a_n=0$ for $n<0$.  

Note now that the above equations give that $c_s=0$ for $s=\ell+m, \ldots ,\ell+2m$, hence 
setting $P(z)=c_0+c_1z+c_{\ell+m-1}z^{\ell+m-1}$, we have $\deg P\le \ell+m-1$, $\deg Q\le m$ and $\ord(P(z)-Q(z)f(z))\ge \ell+2m+1$, as required. The converse assertion is proved by simply reversing the argument. 
\end{proof}

In the sequel we shall let $\Delta_m(f)$ denote the Hankel determinant $\Delta_{0,m}(f)$. We have the following consequence.
\begin{corollary}\label{cor:H1}
Let $m$ and $d$ be nonnegative integers. If $\Delta_{m+i}(f)=0$ for $0\leq i\leq d$ then
there exist $P(z),Q(z)\in k[z]$ with $\deg P\leq m-1$,
$Q(0)\neq 0$, $\deg Q\leq m$, $\gcd(P,Q)=1$, and
$f(z)-P(z)/Q(z)=O(z^{m+d+1})$.
\end{corollary}
\begin{proof}
For $0\leq i\leq d$, Lemma~\ref{lem:H1} gives $P_i(z),Q_i(z)\in k[z]$ with $\deg P_i\leq m+i-1$, $Q_i\neq 0$, $\deg Q_i\leq m+i$ and $P_i(z)-Q_i(z)f(z)=O(z^{2m+2i+1})$. Therefore:
$$P_iQ_{i+1}-P_{i+1}Q_i=O(z^{2m+2i+1}).$$
Since the left-hand side has degree at most $2m+2i$, we have $P_iQ_{i+1}=P_{i+1}Q_i$ for $0\leq i\leq d-1$. We have
\begin{equation}\label{eq:f-P_0/Q_0}
f(z)-P_0(z)/Q_0(z)=f(z)-P_d(z)/Q_d(z)=\frac{1}{Q_d(z)}O(z^{2m+2d+1})=O(z^{m+d+1}).
\end{equation}
Express $\displaystyle\frac{P_0}{Q_0}=\frac{P}{Q}$ with $\gcd(P,Q)=1$.
Let $a=\ord Q_0$ then we have $a\leq \ord P_0$; otherwise the order function of the LHS of \eqref{eq:f-P_0/Q_0} has a negative value. This implies $Q(z)$ divides $Q_0(z)/z^a$. Hence $Q(0)\neq 0$ and this finishes the proof. 
\end{proof}

By letting $d\to\infty$ in Corollary~\ref{cor:H1}, we have the following well-known criterion for rationality by Kronecker:
\begin{corollary}[Kronecker's criterion]
The power series  $f(z)$ is in $k(z)$ if and only if $\Delta_n(f)=0$ for all sufficiently large $n$.
\end{corollary}

%

\subsection{D-finite series and rational approximation}
Throughout this subsection, we let $k$ be a field of characteristic zero and let $f(z)\in k[[z]]$ be a (nonzero) D-finite power series over $k(z)$; explicitly:
\begin{equation}\label{eq:operator L}
\scrL f=0,\qquad \scrL=A_0(z)D^p+A_1(z)D^{p-1}+\ldots +A_p(z),\qquad D=\frac{d}{d z},
\end{equation}
where  $p={\rm dim}_{k(z)}{\rm Span}\{f,f',f'',\ldots \}$ and $A_i\in k[z]$
for $0\leq i\leq p$ with $A_0\neq 0$. Let $\delta=\max\{\deg(A_i):\ 0\leq i\leq p\}$.
\begin{lemma}\label{lem:matrix B}
Let $n$ be an integer with $n\geq p-1$, let $P_0(z),\ldots ,P_n(z)$ be polynomials in $k[z]$ of degree at most $n$, and let $$g(z)= \sum_{i=0}^n P_i(z) f^{(i)}(z).$$ 
If
$${\rm dim}_{k(z)}{\rm Span}\{g,g',g'',\ldots \}=p,$$ 
then
there is a $p\times p$ matrix $B(z)$ with rational function entries whose numerators and denominators are bounded in degree by $(2\delta+1)np$ such that $$B(z)[g(z),g'(z),\ldots, g^{(p-1)}(z)]^T = [f(z),f'(z),\ldots ,f^{(p-1)}(z)]^T.$$
\end{lemma}
\begin{proof}
We have 
$f^{(p)}(z) = -\displaystyle\sum_{i=0}^{p-1} (A_{p-i}(z)/A_0(z)) f^{(i)}(z)$.  Then a straightforward induction shows that for $j\ge 0$ we have
\begin{equation}\label{eq:f^(p+j)}
f^{(p+j)}(z) = \sum_{i=0}^{p-1} (A_{p-i,j}(z)/A_0(z)^{j+1}) f^{(i)}(z),
\end{equation}
where the $A_{p-i,j}$'s are polynomials of degree at most $\delta(j+1)$.

Let $j\in\{0,\ldots,p-1\}$, we have 
$$g^{(j)}= D^j\left(\sum_{k=0}^n P_kf^{(k)}\right)=\sum_{k=0}^{n}D^j (P_kf^{(k)})=\sum_{k=0}^n\sum_{\ell=0}^j \binom{j}{\ell}(D^{j-\ell}P_k)\cdot f^{(k+\ell)}.$$
For this last double sum, when $k+\ell\geq p$ we use \eqref{eq:f^(p+j)} and the inequaltiy $\deg(D^{j-\ell}P_k)\leq n$ to express $(D^{j-\ell}P_k)\cdot f^{(k+\ell)}$ as
$$\sum_{i=0}^{p-1}\frac{U_i(z)}{A_0(z)^{k+\ell-p+1}}f^{(i)}$$
where $U_i(z)$ has degree at most $n+\delta(k+\ell-p+1)$. With the extreme case $(k,\ell)=(n,j)$ in mind, we may take $A_0(z)^{n+j-p+1}$ as a common denominator and express:
$$g^{(j)}=\sum_{k=0}^n\sum_{\ell=0}^j \binom{j}{\ell}(D^{j-\ell}P_k)\cdot f^{(k+\ell)}=\sum_{i=0}^{p-1} C_{p-i,j}(z)/A_0(z)^{n+j-p+1} f^{(i)}(z),$$ where $C_{p-i,j}$ has degree at most $n + \delta(n+j-p+1)$.
Then
$$[g(z),g'(z),\ldots, g^{(p-1)}(z)]^T = C(z)[f(z),f'(z),\ldots ,f^{(p-1)}(z)]^T,$$ where $C(z)$ is the $p\times p$ matrix whose $(i,j)$-entry is
$C_{p-j,i}(z)/A_0(z)^{n+i-p+1}$ for $0\leq i,j\leq p-1$.  Observe that $C(z)$ is invertible, since otherwise there would be a nonzero row vector $w(z)\in k(z)^{p}$ such that $w(z) C(z)=0$, which would give that $g,g',\ldots ,g^{(p-1)}$ are linearly dependent over $k(z)$, contradicting the
assumption that
$$ {\rm dim}_{k(z)}{\rm Span}\{g,g',g'',\ldots \}=p.$$ 

The determinant of $C(z)$ has the form:
$$\det(C(z))=\frac{P(z)}{A_0(z)^{p(n-p+1)+p(p-1)/2}}$$
where $P(z)\in k[z]$ with $\deg P\leq np+\delta p(n-p+1)+\delta p(p-1)/2$. For $0\leq i,j\leq p-1$,
we remove the row $[C_{i,0},C_{i,1},\ldots,C_{i,p-1}]$ and the column $[C_{0,j},C_{1,j},\ldots,C_{p-1,j}]^T$ from the matrix $C(z)$ then the determinant of the resulting matrix has the form:
$$\frac{P_{i,j}(z)}{A_0(z)^{(p-1)(n-p+1)+p(p-1)/2-i}}$$
where $P_{i,j}(z)\in k[z]$ with $\deg P_{i,j} \leq n(p-1)+\delta(p-1)(n-p+1)+\delta((p(p-1)/2)-i)$. Let $B(z)$ be the inverse of the matrix $C(z)$. We have that each entry of $B(z)$ is a rational function 
whose numerator and denominator are bounded in degree by:
$$np+\delta p(n-p+1)+\delta p(p-1)/2 \leq (2\delta+1)np.$$  
\end{proof}

We conclude this subsection with a rational approximation principle.  We recall that we are assuming our base field $k$ has characteristic zero, which is a needed hypothesis in the following result.  
\begin{lemma}\label{lem:f=P/Q}
Suppose that $P,Q\in k[z]$ are polynomials  of degree $\le \ell$ such that $Q(0)\neq 0$ and $P-Qf=O(z^N)$, where 
$N>(p+2)\ell+\delta+p$. Then  we have $\scrL(P/Q)=0$. 
Moreover, there is an effectively computable subset $T$ of $\N$ with $|T|\le p$, depending only on $\scrL$, such that if $\ord (P-Qf)\not\in T$ then we must have in fact $f=P/Q$.
\end{lemma}

\begin{proof}  
If $Qf=P$ we are done, so assume this is not the case. Since $Q(0)\neq 0$ we may write $\displaystyle f-P/Q=z^mR$ with  
$m=\ord(P-Qf)$ and $R\in k[[z]]$ satisfying $R(0)\neq 0$.

We have $\scrL(P/Q)=-\scrL(z^mR)$. The right-hand side is a power series of order $\ge m-p$. The left-hand side is a rational function whose numerator degree  is bounded above by  $ (p+2)\ell +\delta$.  This numerator  must be divisible by $z^{m-p}$, thus must vanish since $m\ge N>(p+2)l+\delta+p$ by assumption.  This proves $\scrL(P/Q)=0$.

For the second assertion, we use that $\scrL(z^mR)=0$ as follows. Note that the term of lowest order in  $A_i(z)D^{p-i}(z^mR(z))$ is  of the form $$R(0)\alpha_i(p-i)!{m\choose p-i}z^{m-p+i+e_i},$$ where  $\alpha_iz^{e_i}$ is the  term of lowest order in $A_i(z)$ (i.e. $e_i=\ord A_i$).  Since the sum of the  terms  of lowest order in $$\displaystyle\sum_{i=0}^pA_i(z)D^{p-i}(z^mR)$$ vanishes, the sum of such terms with the property that $m-p+i+e_i$ assumes the minimum value, say equal to $\mu$,  must also vanish.

Let $I\subseteq \{0,1,\ldots ,p\}$ be  the subset of indices $i$ such that $m-p+i+e_i=\mu$; note that,  for $m\ge p$,  $I$ does not depend on $m$ and corresponds to $i+e_i$ being minimal. We then have 
$$\displaystyle\sum_{i\in I}\alpha_i(p-i)!{m\choose p-i}=0.$$  Since $\characteristic(k)=0$, the left side is a nonzero polynomial in $m$, of degree at most $p$, depending explicitly on the $A_i(z)$.  This polynomial has at most $p$ roots, and if $m$ is outside the set of these roots we have a contradiction, which in fact proves that $f=P/Q$. 
\end{proof}

\subsection{P\'olya's inequality}
In this subsection, we let $G\subsetneq \C$ be a simply connected domain containing $0$ with conformal radius $\rho>1$ from the origin. This means
we have a (unique) conformal map $\varphi$ from $G$ onto
the disk $D(0,\rho)$ with $\varphi(0)=0$ and $\varphi'(0)=1$. Let $r\in (1,\rho)$ and let
$\Gamma=\varphi^{-1}(\partial D(0,r))$. For a (complex-valued) continuous function $f$ on
$G$, let
$$\Vert f\Vert_{\Gamma}=\max\{\vert f(z)\vert:\ z\in \Gamma\}.$$
A more general version of P\'olya's inequality treats functions with
isolated singularities, the following version
for analytic functions is sufficient for our purpose:
 
\begin{theorem}[P\'olya's inequality]\label{thm:Polya}
Let $f(z)$ be an analytic function on $G$ with the Taylor series $f(z)=\sum_{n=0}^{\infty} a_n z^n$ at $0$.  Then there 
exists a constant $C>0$ depending only
on the data $(G,r)$ such that 
$$\vert \Delta_n(f)\vert\le (n+1)! (C\Vert f\Vert_{\Gamma})^{n+1}r^{-n(n+1)} \text{for every $n\in\N_0$,}$$ where 
we recall that $\Delta_n(f)=\Delta_{0,n}(f)$ is defined in Subsection~\ref{subsec:Hankel}. 
\end{theorem}
\begin{proof} The more general version in which $f$ can have isolated singularities is proved in \cite[pp.~121--124]{Bie55_AF}. 
\end{proof}

For the proof of Theorem~\ref{thm:main thm}, we need the following estimate when applying P\'olya's inequality for certain auxiliary functions $g$ constructed from $f$:
\begin{corollary}\label{cor:cor of Polya's inequality}
Let $f(z)$ be an analytic function on $G$ with the Taylor series $f(z)=\sum_{n=0}^{\infty} a_n z^n$ at $0$. There exist  constants $C_1>0$
depending on $(G,r,f)$ and $C_2\ge 1$ depending 
on $(G,r)$ such that 
for every $d\ge 0$ and $c_0,\ldots ,c_d\in \mathbb{C}$, we have
$$\vert\Delta_n(g)\vert \le (n+1)! C_1^{n+1} \left( \sum_{j=0}^d |c_j| j! C_2^{j}\right)^{n+1} r^{-n(n+1)}$$ for every $n\in \N_0$, where
$g(z):= \sum_{j=0}^d c_j z^j f^{(j)}(z)$.
\end{corollary}
\begin{proof}
In order to apply Theorem~\ref{thm:Polya}, we need an upper bound for $\Vert g\Vert_{\Gamma}$ as follows. Let $r'=(r+\rho)/2 \in (r,\rho)$ and let $\Gamma'=\varphi^{-1}(\partial D(0,r'))$. Let $$\kappa:=\min\{\vert z-z'\vert:\ z\in\Gamma,\ z'\in\Gamma'\}$$ denote the distance between $\Gamma$ and $\Gamma'$. For $z_0\in \Gamma$, Cauchy's formula gives:
$$f^{(j)}(z_0) = \frac{j!}{2\pi i} \int_{\Gamma'} f(z)/(z-z_0)^{j+1} \, dz$$ and so
$$\vert f^{(j)}(z_0)\vert \le  (2\pi)^{-1} j! \Vert f\Vert_{\Gamma'} \kappa^{-j-1} {\rm length}(\Gamma'),$$
where we may, for example, use the parametrization $\psi:\ [0,1]\rightarrow \Gamma'$ given by $\psi(t)=\varphi^{-1}(r'\exp(2\pi it))$ and have ${\rm length}(\Gamma')=\int_{0}^1\vert\psi'(t)\vert\,dt$. 
Consequently, $$\Vert g\Vert_{\Gamma}\leq \displaystyle C_3 \left( \sum_{j=0}^d |c_j| j! M^j \kappa^{-j-1}\right)$$ for some constant $C_3$ depending only on $(G,r,f)$, where $M=\Vert z\Vert_{\Gamma}$.   Let $C$ be the resulting constant in Theorem~\ref{thm:Polya}, we now take $C_1=CC_3\kappa^{-1}$ and $C_2=\max(1,M\kappa^{-1})$ then apply Theorem~\ref{thm:Polya} to finish the proof. 
\end{proof}

\section{Proof of Theorem~\ref{thm:main thm}}\label{sec:proof}
In this section we prove our main result. We begin with a simple lemma.
\begin{lemma} Let $L$ and $n$ be natural numbers with $n>4L$ and let $i_1,\ldots ,i_L$ be integers with $$1\le i_1< i_2 < \cdots < i_L \leq n.$$  Then there exist
$c_0,\ldots ,c_{2L}\in \mathbb{Z}$, not all zero, with $\vert c_i\vert < n {n\choose 2L}$ for $i=0,\ldots ,2L$ such that
$$\sum_{i=0}^{2L} c_i {m\choose i} = 0$$ for $m=i_1,\ldots ,i_L$.
\label{lem:Siegel}
\end{lemma}
\begin{proof}
Since $n>4L$, the maximum of ${m\choose i}$ with $i\le 2L$ and $m\in \{i_1,\ldots ,i_L\}$ is ${n\choose 2L}$.  Then by Siegel's lemma \cite[p.~72]{BG06_HI}, there exist $c_0,\ldots ,c_{2L}\in \mathbb{Z}$ not all zero, with $$\vert c_i\vert< \left((2L+1)\cdot{n\choose 2L}\right)^{L/(L+1)} < n {n\choose 2L}$$ such that $$\sum_{i=0}^{2L} c_i {m\choose i} = 0$$ for $m=i_1,\ldots ,i_L$.
\end{proof}

\begin{definition} {\em We say that a polynomial $P(z)\in\Q[z]$ is \emph{integer-valued} if $P(n)$ is an integer for every $n\in \Z$.}
\end{definition} 
\begin{theorem} Let $f(z)=\displaystyle\sum_{n\geq 0} a_n z^n\in \mathbb{Q}[[z]]$ be a power series that converges on the open unit disk and write $d_n=\den(a_n)$. Then one of the following three statements must hold:
\begin{enumerate}
\item $f(z)$ admits the unit circle as a natural boundary;
\item there is some $\beta>1$, depending only upon $f(z)$, such that for every set $S\subseteq\mathbb{N}$ of zero density, we have $\lcm\{d_i \colon i\le n, i\not\in S\} >\beta^n$ for all sufficiently large $n$;
\item for every positive constant $C$, there exist a positive integer $n$, an integer-valued polynomial $P(z)$ of degree at most $n$, and co-prime polynomials $u(z)$ and $v(z)$  of degree at most $n$ with $v(0)\neq 0$ such that 
$${\rm ord}\left(\sum_{j\ge 0} P(j) a_j z^j - u(z)/v(z)\right) > Cn.$$
\end{enumerate}
\label{thm:trichotomy}
\end{theorem}
\begin{proof} Suppose that $f(z)$ does not admit the unit circle as a natural boundary. Let $D$ be a simply connected region of the plane that properly contains the open unit disc such that $f(z)$ can be extended to an analytic function in 
$D$.  Let $r\in (1,\rho)$, where $\rho>1$ is the conformal radius (from the origin) of $D$. For the moment, let $C$ be a fixed positive integer.

Let $S$ be a subset of the natural numbers of zero density and for $m\ge 1$, let $S(m)$ denote the set of natural numbers $j\le m$ with $j\in S$.  Let $n$ be a natural number and let $i_1<\cdots <i_L\le Cn$ denote the elements of $S(Cn)$.  Then by assumption $L/n\to 0$ as $n\to \infty$ and we assume that $n$ is large enough that $2L<n$.

By Lemma \ref{lem:Siegel} there exist
$c_0,\ldots ,c_{2L}\in \mathbb{Z}$, not all zero, with $\vert c_i\vert< Cn {Cn\choose 2L}$ for $i=0,\ldots ,2L$ such that
$$\sum_{i=0}^{2L} c_i {m\choose i}=0$$ for $m=i_1,\ldots ,i_L$.  Then we let
$g_n(z) = \sum_{j=0}^{2L} c_j z^j f^{(j)}(z)/j!$.   We can express 
$$g_n(z) = \sum_{j\geq 0} P_n(j)a_j z^j$$ 
for the integer-valued polynomial $P_n(z):=\displaystyle\sum_{i=0}^{2L}c_i\binom{z}{i}$ of degree at most $2L< n$ and $P_n(i_1)=\cdots = P_n(i_L)=0$.  Then by Corollary~\ref{cor:cor of Polya's inequality}, there exist constants $A>0$ and $\alpha\ge 1$, depending only on the data $(f,D,r)$, such that for $m\in \{n,\ldots ,Cn\}$ we have
\begin{eqnarray*} \vert \Delta_m(g_n(z))\vert & \le & (m+1)! A^{m+1} \left( \sum_{j=0}^{2L} |c_j|  \alpha^{j}\right)^{m+1} r^{-m(m+1)} \\
& <&  (m+1)! A^{m+1} \left(Cn^2 {Cn\choose 2L} \alpha^{2L}\right)^{m+1} r^{-m(m+1)}.\end{eqnarray*}

Let $d(m)=\lcm\{d_i \colon i\le m, i\not\in S\} $.
By definition, $d(m)^{m+1} \Delta_m(g_n(z))$ is an integer for $m=n,\ldots ,Cn$ and hence we must have either $\Delta_m(g_n(z))=0$ for all $m$ in this range, or there is some $m\in \{n,\ldots ,Cn\}$ such that \begin{equation}\label{eq:1/d(m)^(m+1)}
1/d(m)^{m+1} \leq \vert \Delta_m(g_n(z))\vert <  (m+1)! A^{m+1} \left(Cn^2 {Cn\choose 2L} \alpha^{2L}\right)^{m+1} r^{-m(m+1)}.
\end{equation}
We fix $r_0\in (1,r)$ and then fix a small $\epsilon>0$ both of which will be specified later. We use Stirling's estimate $\log k!=k\log k-k+O(\log k)$. When $n$ is large, we have $2L<\epsilon n$ since $S$ has zero density and  then
\begin{align*}
\log\binom{Cn}{2L}&< Cn\log Cn-\epsilon n\log\epsilon n-(C-\epsilon)n\log(C-\epsilon)n+O(\log n)\\
&=(C\log C-\epsilon\log\epsilon-(C-\epsilon)\log(C-\epsilon))n+O(\log n).
\end{align*}
As $\epsilon\to 0$, we have $C\log C-\epsilon\log\epsilon-(C-\epsilon)\log(C-\epsilon)\to 0$. Therefore we may choose a sufficiently small $\epsilon$ such that for every sufficiently large $n$, we have:
\begin{equation}\label{eq:CA...<r_0^n}
    CA n^2 {Cn\choose 2L} \alpha^{2L} < r_0^n.
\end{equation}
Combining \eqref{eq:1/d(m)^(m+1)} and \eqref{eq:CA...<r_0^n}, we have  
$$1/d(m)^{m+1} < (m+1)!  (r_0/r)^{m(m+1)}$$ for all large $n$ with $m\in \{n,\ldots ,Cn\}$ such that $\Delta_m(g_n(z))\neq 0$. 

There are now two cases.  First, if for every positive integer $C$, there is a set $S\subset \N$ of zero density and a large integer  $n$ such that $\Delta_m(g_n(z))=0$ for all $m\in \{n\ldots ,Cn\}$ where $g_n$ is constructed as above.  In this case we see that (iii) holds by Corollary~\ref{cor:H1}.
Second, if there is some $C$ such that for every set $S\subset\N$ of zero density and for all large $n$ we have $\Delta_m(g_n(z))\neq 0$ for some $m\in \{n,\ldots ,Cn\}$ then we have that
$$1/d(m)^{m+1} < (m+1)!  (r_0/r)^{m(m+1)}.$$
But this implies:
 \begin{equation}\label{eq:r/r_0}
 d(Cn) \geq d(m) > (r/r_0)^m (m+1)!^{-1/(m+1)} >  \frac{(r/r_0)^m}{m+1}\geq \frac{(r/r_0)^n}{Cn+1}
\end{equation}
for all sufficiently large $n$. Now we may fix any $\beta\in (1,r^{1/C})$ at the beginning and choose $r_0\in (1,r)$ such that $(r/r_0)^{1/C}>\beta$, then \eqref{eq:r/r_0} gives
 $d(n)>\beta^n$ for all sufficiently large $n$. This means (ii) holds and we finish the proof.
\end{proof}

We have an elementary result relating a lower bound on
$\lcm\{\den(a_i):\ i\leq n\}$ to a lower bound on $\den(a_i)$:
\begin{lemma}\label{lem:kappa from beta}
Let $(d_n)_{n\geq 0}$ be a sequence of positive integers. Suppose there exists $\beta>1$ such that for every set $S\subset \N$ of zero density, we have:
$$\lcm\{d_i:\ i\leq n,\ i\notin S\}>\beta^n$$
for all sufficiently large $n$. Let $\kappa\in (0,\log \beta)$, then the set
$\{n:\ d_n\geq \kappa n\}$ has positive density.
\end{lemma}
\begin{proof}
Put $S=\{n:\ d_n\geq \kappa n\}$ and suppose that $S$ has zero density. Let $n$ be a large integer. The given assumption gives:
$$\lcm\{d_i:\ i\leq n,\ i\notin S\}>\beta^n.$$
On the other hand, the definition of $S$ gives:
$$\lcm\{d_i:\ i\leq n,\ i\notin S\} \leq \lcm\{m:\ 1\leq m\leq \kappa n\} = e^{\kappa n+o(n)},$$
where the second equality follows from the Prime Number Theorem \cite[Chapter~6]{MV06_MN}. Since $\kappa<\log\beta$, we have a contradiction when $n$ is large. Therefore $S$ must have positive density.
\end{proof}

\begin{proposition} Let $f(z)=\displaystyle\sum_{n\geq 0} a_n z^n\in \Qbar[[z]]$ with the property that there exists a nonzero
$P(z)\in\Qbar[z]$ such that $g(z)= \displaystyle\sum_{n\geq 0} P(n) a_n z^n$ is the power series expansion of a rational function.  Then at least one of the following holds: 
\begin{itemize}
	\item [(i)] $h(a_n)\gg n$ on a set of positive density.
	\item [(ii)] There exist a positive integer  $s$ and polynomials $Q_0,\ldots,Q_{s-1} \in \Qbar[z]$ such that
	$P(sn+i)a_{sn+i}=Q_i(n)$ for every $i\in\{0,\ldots,s-1\}$ and sufficiently large $n$.
\end{itemize} 
\label{prop:rat}
\end{proposition}
\begin{proof}
Suppose that (i) does not hold. We may assume that $g(z)$ is not a polynomial since otherwise $a_n=0$ for all sufficiently large $n$ and (ii) holds trivially. Write $g(z)=u(z)/v(z)$ for coprime polynomials $u,v\in\Qbar[z]$ with
$v(0)=1$ and $v$ is non-constant.
Write $v(z)=(1-\alpha_1 z)^{m_1}\cdots (1-\alpha_q z)^{m_q}$ with $q\ge 1$ and $\alpha_1,\ldots ,\alpha_q$ are distinct nonzero algebraic numbers. For all sufficiently large $n$, we can express
$$P(n)a_n=R_1(n)\alpha_1^n+\ldots+R_q(n)\alpha_q^n$$
where the $R_i$'s are nonzero elements of $\Qbar[z]$ \cite[Theorem~4.1.1]{Sta11_EC1}. First, we prove that all the $\alpha_i$'s are roots of unity. Restricting $n$ to appropriate arithmetic progressions if necessary, we may assume that $\alpha_i/\alpha_j$ is not a root of unity for $i\neq j$. Suppose there is a place $w$ of $\Qbar$ such that 
$$M:=\displaystyle\max_{1\leq i\leq q}\vert \alpha_i\vert_w>1.$$ Let $M_0\in (1,M)$ and $\mathcal{E}=\{n:\ \vert a_n\vert_w\leq M_0^n\}$. First \cite[Proposition~2.3]{KMN19_AP} gives that for $n\in \mathcal{E}$, 
the $R_i(n)\alpha_i^n$'s satisfy one of finitely many non-trivial linear dependence relations. Then \cite[Proposition~2.3]{KMN19_AP}
yields that there are only finitely many $n$ for each linear relation. Overall, we have that $\mathcal{E}$ is finite or in other words
$\vert a_n\vert_w > M_0^n$ for all but finitely many $n$. This gives (i), contradicting the assumption at the beginning of the proof. Hence we have $\vert \alpha_i\vert_w\leq 1$ for every place $w$ and for $0\leq i\leq q$. Therefore all the $\alpha_i$'s are roots of unity.
Let $s\in\N$ such that $\alpha_i^s=1$ for every $i$. Then we have polynomials $Q_i(z)$'s for $i\in \{0,\ldots ,s-1\}$ satisfying (ii).

\end{proof}

\begin{theorem}\label{thm: irrational D-finite and lcm d_i} 
Let $f(z)=\sum a_n z^n \in \mathbb{Q}[[z]]$ be a D-finite power series that converges on the open unit disk such that for every nonzero $P(z)\in\Q[z]$ the series
$\displaystyle\sum_{n} P(n)a_nz^n$ is not a rational function. Put $d_n=\den(a_n)$ for every $n$. There exists $\beta>1$ such that for every set $S\subset\N$ of zero density, we have
$$\lcm\{d_i:\ i\leq n,i\notin S\}>\beta^n$$
for all sufficiently large $n$.
\end{theorem}
\begin{proof}
Among all the nonzero integer-valued polynomials $P(z)$, we pick one such that
$${\rm dim}_{\mathbb{C}(z)}{\rm Span}\{ t(z),t'(z),t''(z),\ldots \}\ \text{is minimal,}$$ where 
$t(z)=
\sum_{n\ge 0} a_n P(n)z^n$.   Then it suffices to prove the result for $t(z)$ since $\den(P(n)a_n)$ divides $\den(a_n)$ for every $n$.  Thus for the remainder of the proof, we replace $f(z)$ by $t(z)$ and let
$$p= {\rm dim}_{\mathbb{C}(z)}{\rm Span}\{ f(z),f'(z),f''(z),\ldots \}.$$  
We have $p>1$ since $f(z)\notin\Q(z)$ thanks to the assumption in the statement of the theorem. Then by the above construction, $f(z)=\sum a_n z^n$ satisfies that 
for every nonzero integer-valued $P(z)$ we have
$p={\rm dim}_{\mathbb{C}(z)}{\rm Span}\{F(z),F'(z),F''(z),\ldots \}$, where $F(z)=\sum_{n\geq 0} a_n P(n) z^n$.
Let $$\scrL=\displaystyle\sum_{i=0}^p A_iD^{p-i}$$ be the differential operator as in \eqref{eq:operator L} such that $\scrL(f)=0$ and let $$\delta=\displaystyle\max_{0\leq i\leq p}\deg(A_i),$$ as before. 
Since $f(z)$ is D-finite it does not admit the unit circle as a natural boundary. We assume that the conclusion of this theorem does not hold and arrive at a contradiction.

By Theorem \ref{thm:trichotomy}, we have that for every positive integer $C$, there is a natural number $n$, a nonzero integer-valued polynomial $P(z)$ of degree at most $n$ such that
$$g(z)=\sum a_n P(n) z^n$$ 
has the property that there there exist co-prime polynomials $u(z),v(z)$ of degree at most $n$ with $v(0)\neq 0$ such that 
\begin{equation}\label{eq:ord vg-u}
{\rm ord}(g(z)-u(z)/v(z))>Cn.
\end{equation}

By Lemma~\ref{lem:matrix B}, there are polynomial $r_0(z),\ldots ,r_{p-1}(z), s(z)$ of degree at most 
$(2\delta+1)np$, with $s(z)$ nonzero, such that 
$$s(z)f(z) = \sum_{i=0}^{p-1} r_i(z) g^{(i)}(z).$$
Now ${\rm ord}(g^{(i)}(z) - (u(z)/v(z))^{(i)}) > Cn-i$ by \eqref{eq:ord vg-u} and so 
$${\rm ord}\left(s(z)f(z) - \sum_{i=0}^{p-1} r_i(z) (u/v)^{(i)}(z)\right)>Cn-p.$$  
This gives
${\rm ord}(f(z)-\Theta(z)) > Cn-p-(2\delta+1)pn$, where 
$$\Theta(z)= s(z)^{-1} \sum_{i=0}^{p-1} r_i(z) (u/v)^{(i)}(z).$$ 
Notice that $\Theta(z)$ is a rational function whose numerator and denominator have degree bounded by $(2\delta+1)np + np$.  Taking $C$ large enough, we see that $f(z)$ is necessarily rational by Lemma~\ref{lem:f=P/Q}, contradiction.
\end{proof}

\begin{lemma}
\label{lem:fabry} Let $S\subset\N$ with zero density. If $f(z)$ has radius of convergence $r\in (0,\infty)$ and if $f(z) = \sum_{n\in S} a_n z^n +g(z)$ where $g(z)$ has radius of convergence strictly greater than $r$, then $f(z)$ admits the circle of radius $r$ as a natural boundary.
\end{lemma}
\begin{proof} 
Let $f_0(z)=\sum_{n\in S} a_n z^n$.  Then $f_0(z)$ has radius of convergence $r$ and since $S$ has zero density, by Fabry's gap theorem (see \cite{Tur47_Ot}), $f_0(z)$ admits the circle of radius $r$ as its natural boundary.  Thus $f(z)$ does too, since $g(z)$ is analytic in a circle of radius $r+\epsilon$ centred at zero, for some $\epsilon>0$.
\end{proof}

\begin{proposition}
Let $K$ be a number field, let $f(z)=\sum a_n z^n \in K[[z]]$, and let $r$ be the radius of convergence of $f(z)$. 
Then the following hold:
\begin{itemize}
\item [(a)] if $r\in (1,\infty)$ and $f(z)$ does not admit the circle of radius $r$ as a natural boundary then there is a positive constant $\kappa$ such that $h(a_n)\ge \kappa n$ for $n$ in a set of positive density; 
\item [(b)] if $r\in (0,1)$ and and $f(z)$ does not admit the circle of radius $r$ as a natural boundary then there is a positive constant $\kappa'$ such that $h(a_n)\ge \kappa' n$ for $n$ in a set of positive density; 
\item [(c)] if $r\in \{0,\infty\}$, $f(z)$ is D-finite and $f(z)$ is not a polynomial then $h(a_n)=O(n\log n)$ and there is a positive constant $\kappa$ such that $h(a_n)>\kappa n\log n$ on a set of positive density.
\end{itemize}
\label{prop:when r neq 1}
\end{proposition}
\begin{proof} If $r\in (1,\infty)$, then there is some $c<1$ such that $\vert a_n\vert < c^n$ when $n$ is sufficiently large. Therefore when $a_n\neq 0$, we have $\vert a_n^{-1}\vert>(1/c)^n$ is exponentially large.  Hence there is a positive constant $\kappa$ such that for $n$ sufficiently large we have $h(a_n)=h(a_n^{-1})\ge \kappa n$ whenever $a_n\neq 0$.  Since $f(z)$ does not admit the circle of radius $r$ as a natural boundary, by Lemma~\ref{lem:fabry} we have that the set of $n$ for which $a_n\neq 0$ has positive density.

 If $r\in (0,1)$, there is some $c>1$ such that $\vert a_n\vert>c^n$ for infinitely many $n$.  Pick $b\in (1,c)$ with $1/b>r$ and let $S$ denote the set of $n$ for which $|a_n|>b^n$.  We claim that $S$ has positive density.  To see this, suppose that $S$ has density zero.  Then we can write $f(z) = f_0(z)+f_1(z)$, where 
$f_0(z)=\sum_{n\in S} a_n z^n$ and $f_1(z)=\sum_{n\not\in S} a_n z^n$. Then by construction $f_1$ is convergent in the disc $D(0,1/b)$.  Then $f_0(z)$ must have radius of convergence $r$ since $f(z)$ has radius of convergence $r$ and $f_1$ has radius of convergence strictly greater than $r$.  Since $S$ has zero density, by Lemma~\ref{lem:fabry} we have that $f_0(z)$ admits the circle of radius $r$ as its natural boundary and so $f(z)$ must too since $f_1(z)$ is analytic in $B(0,1/b)$ and $1/b>r$.  This is a contradiction.  Thus we obtain (b). 

It remains to consider the case $r\in \{0,\infty\}$. The property $h(a_n)=O(n\log n)$ has been proved in 
\cite[Theorem~1.6]{BNZ20_DF}.
The power series $f(z)$ is a Gevrey series whose precise Gevrey order is a rational number: there is $s\in\Q$ such that 
$ \sum a_n (n!)^s z^n$ has positive finite radius of convergence (see, for example, \cite{Ram78_DG}). Since $r\in\{\infty,0\}$, we have $s\neq 0$.  We write $s=a/b$ with $a,b$ nonzero integers and $b>0$.  Then by construction $g(z):=\sum_{n\ge 0} a_n^b n!^a z^n$ is D-finite and has positive nonzero radius of convergence.  In particular, by using Lemma~\ref{lem:fabry} as above there exist $r_0,r_1\in (0,\infty)$ with $r_0<r_1$ such that 
 $$r_0^n < a_n^b n!^a < r_1^n$$ for $n$ in a set of natural numbers of positive density.  In particular, if $a>0$ then we see that
 $a_n^{-b}> n!^a/r_1^n$, which gives that $h(a_n)> \kappa n\log n$ for some positive constant $\kappa$ for $n$ in a set of positive density.  Similarly, if $a<0$ then $a_n^b > r_0^n n!^{-a}$ and so the result follows similarly in this case.
\end{proof}

\begin{theorem}\label{thm:slightly stronger}
	Let $f(z)=\sum_{n=0}^{\infty} a_nz^n\in \Qbar[[z]]$ be a D-finite power series and let $r$ be the radius of convergence of $f$. We have:
	\begin{itemize}
		\item [(a)] If $r\in\{0,\infty\}$ and $f$ is not a polynomial then $h(a_n)=O(n\log n)$ for every large $n$ and $h(a_n)\gg n\log n$ on a set of positive density.
		\item [(b)] If $r\notin\{0,\infty\}$ then at least one of the following holds:
		\begin{itemize}		
		\item [(i)] $h(a_n)\gg n$ on a set of positive density.
		
		\item [(ii)] There is a constant $\beta>1$ such that for every set $S\subset \N$ of zero density, we have:
		$$\lcm\{\den(a_i):\ i\leq n,\ i\notin S\}>\beta^n$$
		for all sufficiently large $n$.
					
		\item [(iii)] There exist $s\in\N$ and 
		rational functions $Q_0(z),\ldots,Q_{s-1}(z)\in \Qbar(z)$ such that
		$a_{sn+i}=Q_i(n)$ for every $i\in\{0,\ldots,s-1\}$ and sufficiently large $n$.
		\end{itemize}	
	\end{itemize}
\end{theorem}
\begin{proof}
Part (a) follows from Proposition~\ref{prop:when r neq 1}, we now assume $r\notin\{0,\infty\}$. 
Suppose neither (i) nor (ii) hold. Let $K$ be a number field 
such that $f(z)\in K[[z]]$ and let $\cO$ be its ring of integers.
Let $m:=[K:\Q]$ and let $\{\gamma_1,\ldots,\gamma_m\}$ be an integral basis of
$K$. Write
$\Tr:=\Trace_{K/\Q}$. If there exist a nonzero $P(z)\in\Qbar[z]$ such that $\sum P(n)a_nz^n$ is a rational function then (iii) follows from  Proposition~\ref{prop:rat}. From now on, we assume that there does not exist such a $P$.

For $1\leq i\leq m$, write $$ g_i(z)=\Tr(\gamma_i f(z))=\sum_{n}a_{i,n}z^n\in\Q[[z]],$$ where
$a_{i,n}=\Tr(\gamma_ia_n)$. 
From the basic height inequalities in Subsection~\ref{sec:height} and the assumption that (i) does not hold, for each $i\in\{1,\ldots,m\}$ we do \emph{not} have that $h(a_{i,n})\gg n$
on a set of positive density. By Proposition~\ref{prop:when r neq 1} and the fact that $g_i$ is D-finite, we have that $g_i$ has radius of convergence $1$ for $i\in\{1,\ldots,m\}$.

First, consider the case that for every $i\in\{1,\ldots,m\}$ there exists a nonzero $P_i(z)\in\Qbar[z]$
such that $ \sum_{n} P_i(n)a_{i,n}z^n$ is a rational function. For every $i\in\{1,\ldots,m\}$,  Proposition~\ref{prop:rat} implies the existence of $s_i\in\N$ and $R_{i,j}(z)\in \Qbar(z)$ such that 
$a_{i,s_in+j}=R_{i,j}(n)$ for $j\in\{0,\ldots,s_i-1\}$ and sufficiently large $n$. From elementary field theory, there exist $c_1,\ldots,c_m\in K$ such that for every $x\in K$ we have
\begin{equation}\label{eq:elementary field theory}
x=c_1\Tr(\gamma_1 x)+\cdots+c_m\Tr(\gamma_m x).
\end{equation}
Indeed, by linearity over $\Q$, it suffices to have \eqref{eq:elementary field theory} for $x\in\{\gamma_1,\ldots,\gamma_n\}$. To find the $c_i$'s, we solve
the linear system
$$ M\cdot [c_1,\ldots,c_m]^T=[\gamma_1,\ldots,\gamma_m]^T$$
where $M=(\Tr(\gamma_i\gamma_j))_{1\leq i,j\leq m}$ is invertible \cite[pp.~11]{Neu99_AN}.
Then we can write $a_n=c_1a_{1,n}+\cdots+c_ma_{m,n}$ for every $n$. It now follows that (iii) holds with $s=\lcm\{s_1,\ldots,s_m\}$.

We now consider the case when there is $i_0\in\{1,\ldots,m\}$ such that for every nonzero $P(z)\in\Qbar[z]$ the series
$\displaystyle \sum_{n}P(n)a_{i_0,n}z^n$ is not a rational function. Theorem~\ref{thm: irrational D-finite and lcm d_i} yields $\beta>1$ such that for every set $S\subset \N$ of zero density we have
$$\lcm\{\den(a_{i_0,k}):\ k\leq n,k\notin S\}>\beta^n$$
 for all sufficiently large $n$. This implies
 $$\lcm\{\den(a_k):\ k\leq n,k\notin S\}>\beta^n$$
 for all sufficiently large $n$ since $\den(a_{i_0,k})\mid \den(a_k)$ for every $k$. But since we are assuming that (ii) does not hold, the current case cannot happen and we finish the proof.
\end{proof}

\begin{proof}[Proof of Theorem~\ref{thm:main thm}] By Theorem~\ref{thm:slightly stronger}, it remains to prove part (b) of Theorem~\ref{thm:main thm}. 
Assume that neither (i) nor (ii) of Theorem~\ref{thm:main thm}(b) hold. By Lemma~\ref{lem:kappa from beta}, we have that neither (i) nor (ii) of Theorem~\ref{thm:slightly stronger} hold. Hence there exist $s\in\N$ and rational functions $Q_0,\ldots,Q_{s-1}\in K(z)$ such that $a_{sn+i}=Q_i(n)$ for every $i\in\{0,\ldots,s-1\}$ and sufficiently large $n$. First, we consider the case when $Q_i(z)$ is a polynomial for every $i\in\{0,\ldots,s-1\}$. This means the $a_n$'s for sufficiently large $n$ are quasipolynomial \cite[Section~4.4]{Sta11_EC1}. Then it follows that $f(z)$ is a rational function whose finite poles must be roots of unity \cite[Proposition~4.4.1]{Sta11_EC1}.

We now consider the case when some $Q_{i_0}$ is not a polynomial. Let $\cO$ be the ring of integers of $K$, write $Q_{i_0}(z)=A(z)/B(z)$, where $A$ and $B$ are coprime polynomials with coefficients in $\cO$ and $B$ is non-constant. It is an easy fact that $\den(A(n)/B(n))\gg n^{\deg(B)}$ when $n$ is sufficiently large and we include the proof here for the sake of completeness. Write $\delta_n=\den(A(n)/B(n))$. There exist $U(z),V(z)\in\cO[z]$ and a nonzero $C\in\cO$ such that:
$$U(z)A(z)+V(z)B(z)=C.$$
Then for every $v\in M_K^0$ and $n\in\N$ we have:
\begin{equation}\label{eq:C vs max A,B at v}
\vert C\vert_v\leq \max\{\vert A(n)\vert_v,\vert B(n)\vert_v\}.
\end{equation}

Let $v\in M_K^0$ such that $\vert A(n)\vert_v\geq \vert B(n)\vert_v$. From $\vert \delta_n A(n)/B(n)\vert_v\leq 1$, we
have:
$$\vert C\delta_n\vert_v\leq \vert C\vert_v \cdot\vert B(n)/A(n)\vert_v\leq \vert B(n)\vert_v,$$
where the rightmost inequality follows from \eqref{eq:C vs max A,B at v} and the assumption on $v$.

Let $v\in M_K^0$ such that $\vert A(n)\vert_v<\vert B(n)\vert_v$. We have:
$$\vert C \delta_n\vert_v\leq \vert C\vert_v\leq \vert B(n)\vert_v,$$
where the rightmost inequality follows from \eqref{eq:C vs max A,B at v} and the assumption on $v$.

Therefore we have $\vert C\delta_n\vert_v\leq \vert B(n)\vert_v$ for every $v\in M_K^0$. Combining this with the product formula and the fact that $\delta_n\in\Q$, we have:
$$\delta_n \prod_{v\in M_K^{\infty}} \vert C\vert_v=\frac{1}{\prod_{v\in M_K^0} \vert C\delta_n\vert_v}\geq\frac{1}{\prod_{v\in M_K^0}\vert B(n)\vert_v}=\prod_{v\in M_K^{\infty}}\vert B(n)\vert_v\gg n^{\deg(B)}.$$
Since $\deg(B)\geq 1$, we have that $\den(a_n)\gg n$ on a set of positive density. Since we are assuming that (ii) does not hold, the current case cannot happen and we finish the proof.
\end{proof}

\begin{remark}
Throughout this paper, we obtain lower bounds for $h(a_n)$ for $n$ in a set of positive \emph{upper} density. We explain how to slightly improve some of the previous results. For a subset $S$ of $\N$, its lower density is $$\liminf_{n\to\infty}\frac{\vert S\cap [1,n]\vert}{n}.$$ We can strengthen Theorem~\ref{thm:trichotomy}(ii), Theorem~\ref{thm: irrational D-finite and lcm d_i}, and Theorem~\ref{thm:slightly stronger}(b)(ii) by replacing ``zero density''  by ``zero lower density'' using similar arguments to the earlier proofs. Finally, we can strengthen Theorem~\ref{thm:main thm}(b)(ii) by replacing ``positive upper density'' by ``positive lower density''. 
\end{remark}

 \section{Concluding Remarks}\label{sec:conclude}
We start with a comment on the subtlety of Theorem~\ref{thm:main thm}. In view of Theorem~\ref{thm:main thm}, Theorem~\ref{thm:slightly stronger}, and Lemma~\ref{lem:kappa from beta}, it is natural to ask whether one can strengthen Theorem~\ref{thm:main thm} by replacing the property ``$\den(a_n)\gg n$ 
on a set of positive density'' in (ii) by the property ``$\lcm\{\den(a_k):\ k\leq n,k\notin S\}>\beta^n$
for all sufficiently large $n$ for every $S$ of zero density''. The answer is negative even for the simple example 
$$\log(1+x)=\displaystyle \frac{1}{x}-\frac{1}{2}x^2+\frac{1}{3}x^3-\cdots$$ 
in which $\den(a_n)=n$ for every $n$. A version of the prime number theorem gives 
$\lcm\{1,\ldots,n\}=e^{n+o(n)}$ for all large $n$. On the other hand, even when we relax the condition ``for all large $n$'' to the condition ``for infinitely many $n$'', it is still possible to construct a set $S$ of zero density such that $\lcm\{i:\ i\leq n,i\notin S\}$ fails to have an exponential lower bound:
\begin{proposition}
There exists $S\subset\N$ of zero density such that for every $\beta>1$ the inequality
$$\lcm\{i:\ i\leq n,i\notin S\}>\beta^n$$
holds for only finitely many $n$.
\end{proposition}
\begin{proof}
We wish to thank Andrew Granville 
for the following construction improving the one in an earlier version of the paper. Let $\alpha:\ \N\rightarrow (0,\infty)$ be an increasing function (i.e. $\alpha(m)\leq \alpha(n)$ for $m\leq n$) such that 
$\displaystyle\lim_{n\to\infty}\alpha(n)=\infty$. For the purpose of \cite[Section~1]{BGNS22_AG}, we now construct $S\subset\N$ such that
\begin{equation}\label{eq:o(nalpha(n)/logn}
    \vert S\cap [1,n]\vert=o(\alpha(n)\cdot n/\log n)
\end{equation}
and for every $\beta>1$ there are only finitely many $n$ such that
$$\lcm\{i:\ i\leq n,i\notin S\}>\beta^n.$$

Let $(c_k)_{k\in\N}$ be an increasing sequence of positive numbers  such that 
\begin{equation}\label{eq:c_k properties}
c_{k+1}<2c_k,\ \displaystyle\lim_{k\to\infty} c_k=\infty,\ \log c_k=o(\alpha(2^{k-1})),\ \text{and}\ \log c_k=o(k),
\end{equation}
for example we may take 
$c_1=\min\{1,\alpha(1)\}$ and 
$c_k=\min\{1.9c_{k-1},k,\alpha(2^{k-1})\}$ for $k\geq 2$. For $k\in\N$, let 
$$S_k:=\{2^{k-1}<n\leq 2^k:\ \text{there is a prime $p$ with $2^k/c_k<p\leq 2^k$ and $p\mid n$.}\}.$$

When $k$ is sufficiently large, we have $\log(2^k/c_k)\sim k\log 2$ thanks to \eqref{eq:c_k properties}. Combining this with \cite[Theorem~2.7(d)]{MV06_MN}, we have:
\begin{align}\label{eq:S_k}
    \begin{split}
        \vert S_k\vert \leq \sum_{2^k/c_k<p\leq 2^k} \frac{2^k}{p}&=2^k(\log\log 2^k-\log\log (2^k/c_k)+O(1/\log(2^k/c_k)))\\
        &=2^k\log \frac{k\log 2}{k\log 2-\log c_k} + O(2^k/k)\\
        &=2^k\log\left(1+\frac{\log c_k}{k\log 2-\log c_k}\right)+O(2^k/k)\\
        &=\frac{2^k\log c_k}{k\log 2}+o((2^k\log c_k)/k)
    \end{split}
\end{align}
where the last equality follows from $\log (1+t)=t+o(t)$ where
$t=\log c_k/(k\log 2-\log c_k)$ is small.

Let $n$ be sufficiently large and let $k$ be such that $2^{k-1}<n\leq 2^k$. In the following estimates, the implied constants are independent of $n$ and $k$. From \eqref{eq:S_k}, we have:
\begin{align}
    \begin{split}
        \vert S\cap [1,n]\vert\leq \sum_{j=1}^k\vert S_j\vert\ll \sum_{j=1}^k \frac{2^j\log c_j}{j}\ll \frac{2^k\log c_k}{k}=o(\alpha(n)\cdot n/\log n)
    \end{split}
\end{align}
since $2^k/k=O(n/\log n)$ while $\log c_k=o(\alpha(2^{k-1}))=o(\alpha(n))$.

It remains to estimate $\lcm\{i:\ i\leq n, i\notin S\}$. Let $i\leq n$ and $i\notin S$. Suppose $p$ is a prime dividing $i$. We prove:
\begin{equation}\label{eq:claim}
    \text{Claim:}\ p\leq 2^k/c_k.
\end{equation}
Suppose otherwise $p>2^k/c_k$, then from the definition of $S_k$ and the assumption that $i\notin S_k$, we have $\ell\leq k-1$ such that $2^{\ell-1}<i\leq 2^{\ell}$. And since $i\notin S_{\ell}$, we have
$$p<2^{\ell}/c_{\ell}.$$
But then we would have $2^{\ell}/c_{\ell}>2^k/c_k$, contradicting 
\eqref{eq:c_k properties}. This proves the claim in \eqref{eq:claim}
which, in turn, implies
$$\log\lcm\{i:\ i\leq n,i\notin S\}\ll \sum_{p\leq 2^k/c_k}\log p=2^k/c_k+o(2^k/c_k)=o(n)$$
thanks to the Prime Number Theorem \cite[Chapter~6]{MV06_MN} and \eqref{eq:c_k properties}. This finishes the proof.
\end{proof}

 We now pose some questions related to the work we have obtained.  In Corollary~\ref{cor:cor of main thm} we show that if $f(z)=\sum a_n z^n \in K[[z]]$ is D-finite and not rational, with $K$ a number field, then we have the inequality $$\displaystyle h(a_n)>\displaystyle\frac{1}{[K:\Q]}\log n+O(1)$$ for $n$ in a set of positive density.  Moreover, we showed that there are examples which show this is essentially best possible when $K=\mathbb{Q}$. In the case that $[K:\Q]>1$, however, the bound appears to be suboptimal and so we ask if one can improve the bounds in this case.
 	\begin{question}\label{Q1}
	For $K$ and $f$ as in Corollary~\ref{cor:cor of main thm}, is it true that $h(a_n)>\log n + O(1)$
	on a set of positive density?
	\end{question}
	
	Another natural question is whether one can give a complete classification of the types of growth rates of heights of coefficients of D-finite series that can occur.  To the best of our knowledge, there are really only four types of possible behaviours, and the question below asks if this gives a  taxonomy of possible ``height gaps'' of D-finite series.
	\begin{question}\label{q:classification}
	Let $f(z)\in\Qbar[[z]]$ be D-finite, is it true that one of the following holds:
	\begin{itemize}
	\item [(i)] $h(a_n)=O(n\log n)$ for every $n$ and $h(a_n)\gg n\log n$ on a set of positive density.
	\item [(ii)] $h(a_n)=O(n)$ for every $n$ and $h(a_n)\gg n$ on a set of positive density.
	\item [(iii)] $h(a_n)=O(\log n)$ for every $n$ and $h(a_n)\gg \log n$ on a set of positive density.
	\item [(iv)] $h(a_n)=O(1)$ for every $n$.
	\end{itemize}
	\end{question}
	
	A weaker version of (the affirmative answer to) Question~\ref{q:classification} is the following collection of four statements concerning a D-finite power series $\sum_n a_nz^n\in\Qbar[[z]]$: 
	\begin{itemize}
		\item [(i')] $h(a_n)=O(n\log n)$.
		\item [(ii')] If $h(a_n)=o(n\log n)$ then $h(a_n)=O(n)$.
		\item [(iii')] If $h(a_n)=o(n)$ then $h(a_n)=O(\log n)$.
		\item [(iv')] If $h(a_n)=o(\log n)$ then $h(a_n)=O(1)$.
	\end{itemize}  	
	Statements (i') and (iv') have been established in \cite{BNZ20_DF}. Statement (ii') is similar to the following long standing open problem in the theory of Siegel E-functions \cite{Sie29_U}, \cite[pp.~11--12]{Beu08_E}, \cite[Section~2.1]{FR22_OS}:
	\begin{question}
	Let $f(z)=\sum_{n}a_nz^n\in\Qbar[[z]]$ be D-finite. Suppose that $h(a_0,\ldots,a_n)=o(n\log n)$, is it true that $h(a_0,\ldots,a_n)=O(n)$? Here $h(a_0,\ldots,a_n)$ is taken to be $h([1:a_0:\cdots:a_n])$.
	\end{question}

 \bibliographystyle{amsalpha}
 \bibliography{DfiniteII} 	 

\end{document}